\documentclass{amsart}
\usepackage{amsthm,amstext,amsmath,amscd,amssymb,latexsym}
\usepackage{color}
\usepackage[matrix,arrow]{xy}
\usepackage{amsfonts,enumerate,array}
\usepackage[colorlinks=true]{hyperref}
\usepackage{todonotes}
	\usepackage{verbatim}
\usepackage[toc,page]{appendix}

\newcommand{\PP}{{\mathbb P}}

\newcommand{\II}{{\mathcal I}}

\newcommand{\Bs}{{\rm{Bs}}}

\newcommand{\ls}{{\mathcal{L}}}
\newcommand{\Nn}{{\rm{N}}}
\newcommand{\NE}{{\rm{NE}}}

\newcommand{\paper}{: \begin{it}}
\newcommand{\jour }{, \end{it}}

\newtheorem{theorem}{Theorem}[section]
\newtheorem{lemma}[theorem]{Lemma}
\newtheorem{proposition}[theorem]{Proposition}
\newtheorem{corollary}[theorem]{Corollary}
\newtheorem{conjecture}[theorem]{Conjecture}

\newtheorem{definition}[theorem]{Definition}
\newtheorem{notation}[theorem]{Notation}
\newtheorem{example}[theorem]{Example}

\newtheorem{remark}[theorem]{Remark}

\numberwithin{equation}{section}

\definecolor{mygreen}{rgb}{0.13, 0.55, 0.13}

\title{Positivity of divisors on blown-up projective spaces, I}

\author{Olivia Dumitrescu}
\address{
Olivia Dumitrescu:
Central Michigan University\\
Pearce Hall 209\\
Mt. Pleasant, Michigan 48859, US\\}

\address{and Simion Stoilow Institute of Mathematics\\
Romanian Academy\\
21 Calea Grivitei Street\\
010702 Bucharest, Romania}
\email{dumit1om@cmich.edu}

\author{Elisa Postinghel}
\address{Elisa Postinghel:
Department of Mathematical Sciences, Loughborough University, LE11 3TU, UK}
\email{E.Postinghel@lboro.ac.uk}

\thanks{The first author is a member of 
the Simion Stoilow Institute of Mathematics of the 
Romanian Academy.}



\begin{document}

\begin{abstract} 
We study $l$-very ample, ample and semi-ample
divisors on the blown-up projective space
$\PP^n$ in a collection of points in general position. 
We establish Fujita's conjectures for all ample 
divisors with the number of points bounded above by $2n$ and for an infinite family of 
ample divisors with an arbitrary number of points.
\end{abstract}

\maketitle

\section*{Introduction}

Ample line bundles are fundamental objects in Algebraic Geometry. 
From the geometric perspective, an ample line bundle is one such that some positive multiple of the underlying divisor moves in a linear system that is large enough to give a projective embedding. In numerical terms a divisor is ample if and only if it lies in the interior of the real cone generated by nef divisors (Kleiman). Equivalently, a divisor is ample if it intersects positively every closed integral subscheme (Nakai-Moishezon).
In cohomological terms, an ample line bundle is one such that a twist of any coherent sheaf by some power is generated by  global sections (Serre). Over the complex numbers, ampleness of line bundles is also equivalent to the existence of a metric with positive curvature (Kodaira).

The very ampleness of divisors on blow-ups of projective spaces and other varieties was studied
by several authors, e.g. Beltrametti and  Sommese \cite{bs}, Ballico and Coppens \cite{baco}, Coppens
\cite{coppens,coppens2},  Harbourne \cite{Ha}.
The notion of $l$-very ampleness of line bundles was introduced by Beltrametti, Francia and Sommese \cite{bfs} and $l$-very ample line bundles on del Pezzo surfaces were classified by Di Rocco \cite{dirocco}.

This paper studies ampleness, $l$-very ampleness and further positivity questions of divisors on blow-ups of projective spaces of higher dimension  in an arbitrary number of points in  general position.

The main tools used are the vanishing theorems for divisors on blown-up projective spaces in points in general position that were proved in \cite{dp}. Generalization of these results to the case of points in arbitrary position
are studied in \cite{badp}.

Vanishing theorems  for divisors on blown-up spaces were firstly used in order to give a solution to the corresponding
 interpolation problem, 
namely to compute the
dimension of the linear system of divisors
 on blown-up projective spaces in points in general position.
The case of linear systems whose  base locus consisted only 
of the union of the linear cycles spanned by the points with multiplicity was studied in \cite{bdp1} where, in particular, a formula
for the dimension of all linear systems with $s\le n+2$ points was given. 
The fact that the strict transform of these linear systems via a resolution of the base locus 
is globally generated is proved in a forthcoming paper \cite{dp-pos2}.
Moreover, in \cite{bdp3} a conjectural formula for the dimension 
of all linear systems with $n+3$ points, that takes into account the contributions given 
by the presence in the base locus of the (unique) rational normal curve of degree $n$ through the $n+3$ 
points and the joins of its secants with the linear subspaces spanned by the points, is given.
In \cite{bdp2} linear systems in $\PP^3$  that contain 
the unique quadric through nine points in their fixed locus are studied and, moreover, Nagata-type results are given
for planar linear systems.

In this paper we employ  the vanishing theorems to prove a a number of positivity properties. 
A first application of the vanishing theorems  is the 
description of {\it $l$-very ample divisors}, in particular {\it globally generated} 
divisors and {\it very ample} divisors contained in  Theorem \ref{tva theorem X0}.

Moreover we establish  {\it Fujita's conjectures} for $\PP^n$ blown-up in $s$ points when $s\leq 2n$, Proposition \ref{fujita answer}, and for an infinite family of divisors for arbitrary $s$, with a bound on the coefficients, Proposition \ref{fujita answer2}.

This paper is organized as follows. In Section \ref{notations} we introduce the general construction, notation and some preliminary facts. 
Section \ref{lva section} contains one of the main results of this article, 
Theorem \ref{tva theorem X0}, that concerns  $l$-very ampleness
of line bundles on blown-up projective spaces in an arbitrary number of points in  general position.
In Section \ref{other positivity} we characterize other positivity properties of divisors on  blown-up projective spaces at points such as nefness, ampleness, bigness, and we establish Fujita's conjecture.

\subsection*{Acknowledgments} 
The authors would like to thank the Research Center FBK-CIRM  Trento for the hospitality and financial support during their one month ``Research in Pairs'', Winter 2015. 
We would also like to express our gratitude to Edoardo Ballico for several useful discussions.

\section{Preliminary results and conjectures}\label{notations}
Let $K$ be an algebraically closed field of characteristic zero.
 Let $\mathcal{S}=\{p_1,\dots,p_s\}$ be a collection of $s$ distinct points in $\PP^n_K$ and
let $S$ be the set of indices parametrizing $\mathcal{S}$, with $|S|=s$.
 
Let
\begin{equation}\label{linear system}
\ls:=\ls_{n,d}(m_1,\dots,m_s)
\end{equation}
denote the linear system of degree-$d$ hypersurfaces of $\PP^n$ with multiplicity at least $m_i$ at $p_i$, for $i=1,\dots,s$.

\subsection{The blow-up of $\PP^n$}\label{blow-up in points}
Assume $\mathcal{S}$ consists of  points in general position. 
 We denote by $X_{s}$ the blow-up of $\PP^n$ in the
 points of $\mathcal{S}$ and by $E_i$ the exceptional divisor of $p_i$, for all $i$.
 The Picard group of $X_{s}$ is spanned by the class of a general hyperplane, $H$, and the classes of the
 exceptional divisors $E_{i}$, $i=1,\dots,s$. 

\begin{notation}\label{generality definition}
Fix positive integers $d, m_1, \ldots, m_s$ and define the following divisor on $X_{s}$:
  \begin{equation}\label{def D}
  dH-\sum_{i=1}^{s} m_i E_i\in\textrm{Pic}\left(X_{s}\right).
  \end{equation}
In this paper we denote by $D$  a \textit{general divisor} in $|dH-\sum_{i=1}^{s} m_i E_i|$. Notice that
the global sections of $D$ are in bijection with the elements 
of the linear system $\ls$ defined in \eqref{linear system}.
  \end{notation}
\begin{remark}\label{generality remark}
It is proved in \cite[Proposition 2.3]{CC} and \cite[Lemma 4.2]{CDDGP} that for any divisor $D$ of the form \eqref{def D}, the general member of $|D|$
 has multiplicity equal to $m_i$ at the point $p_i$. It is important to mention here that 
even if it is often omitted in the framework of classical interpolation problems in $\PP^n$,
 the generality hypothesis of the divisor $D$ is always assumed. 
\end{remark}

For any effective divisor on $X_{s}$ denote by $s(d):=s_D(d)$ the number of points in $\mathcal{S}$  
of which the multiplicity equals $d$; this number depending on $\ls$ or, equivalently, $D$. 
We we use the following integer, that was introduced in \cite[Theorem 5.3]{bdp1}:
\begin{equation}\label{b bound}
b=b(D):= \min\{n -s(d), s-n-2\}.
\end{equation}

\begin{theorem}[{\cite[Theorem 5.3]{bdp1}, \cite[Theorem 5.12]{dp}}]\label{vanishing theorems}
Assume that $\mathcal{S}\subset\PP^n$ is a set of points in  general position. Let $D$ be as in \eqref{def D}. Assume that
\begin{equation}\label{cond vanishings}
\begin{split}
& 0\le m_i, \ \forall i\in\{1,\dots,s\},\\ 
&m_i+m_j\le d+1, \ \forall i,j\in\{1,\dots,s\},\ i\ne j, \ (\textrm{if } s\ge2),\\
&\sum_{i=1}^{s} m_i\le nd+\left\{ \begin{array}{ll}
n & \textrm{ if } s\le n+1\\
1 & \textrm{ if } s= n+2\\
b & \textrm{ if }  s\ge n+3\end{array}\right. 
\end{split}
\end{equation}
Then $h^1(X_{s},D)=0$.
\end{theorem}

\subsection{Base locus and effectivity of divisors on the blow-up of $\PP^n$ in points}
\label{base locus and effectivity}

We adopt the same notation as in Subsection \ref{blow-up in points}. 
For any integer 
$0\le r\le  n-1$ and for any multi-index of cardinality $r+1$, 
$$I:=\{i_1,\ldots,i_{r+1}\}\subseteq\{1,\ldots,s\},$$ we define the integer 
$k_{I}$ to be the \emph{multiplicity of containment} of the strict transform in 
$X_{s}$ of the linear cycle spanned by the points parametrized by $I$, $L_{I}\subset\PP^n$, in
 the base locus of  $D$. Notice that under the generality assumption we have $L_I\cong\PP^r$.

\begin{lemma}[{\cite[Proposition 4.2]{dp}}]\label{base locus lemma}
For any effective divisor $D$  as in \eqref{def D},
 the multiplicity of containment in $\Bs(|D|)$ of the strict transform in $X_{s}$ of the linear subspace
 $L_{I}$, with $0\leq r\leq n-1$, is the integer
\begin{equation}\label{mult k} 
k_{I}=k_I(D):=\max \{0,m_{i_1}+\cdots+m_{i_{r+1}}-rd\}.
\end{equation}
\end{lemma}

\vskip.5em

\section{$l$-very ample divisors on $X_{s}$}\label{lva section}

\begin{definition}[\cite{bfs}]\label{definition k-va}
Let $X$ be a complex projective smooth variety.
For an integer $l\ge0$, a line bundle  $\mathcal{O}_X(D)$ on $X$ is said to be \textit{$l$-very ample}, 
if for any $0$-dimensional subscheme $Z\subset X$ of weight $h^0(Z,\mathcal{O}_Z)=l+1$, 
 the restriction map $H^0(X,\mathcal{O}_X(D))\to H^0(Z, \mathcal{O}_X(D)|_Z)$ is surjective.
\end{definition}

We will now recall some of the results obtained in the study of positivity of blown-up surfaces and higher dimensional projective spaces. Di Rocco
 \cite{dirocco} classified  $l$-very ample line bundles on del Pezzo surfaces, namely for $\PP^2$ blown-up at $s\le 8$ points in general position.
For general surfaces, very ample divisors on rational surfaces were considered by Harbourne \cite{Ha}.
De Volder and Laface \cite{DeLa1} classified 
 $l$-very ample divisors, for $l=0,1$, on the blow-up of $\PP^3$ at $s$ points lying on a certain quartic curve. Ampleness and very ampleness properties of divisors on blow-ups at points of higher dimensional 
projective spaces in the  case of points of multiplicity one
were studied by Angelini \cite{angelini}, Ballico \cite{ba} and Coppens \cite{coppens2}. 

 Positivity properties for blown-up $\PP^{n}$ in general points were considered 
by Castravet and Laface. In particular, for small number of points in general position, 
$s\leq 2n$, the semi-ample and nef cones, that we describe in this paper in Corollary \ref{nef cone}, were obtained via a different technique (private communication).

\vskip.3cm

We can describe $l$-very ample line bundles over $X_{s}$, the blown-up 
projective space at $s$ points in  general position,  whose underlying divisor is of the form 
\eqref{def D} $$D=dH-\sum_{i=1}^sm_iE_i,$$
as follows.

Take $l\ge 0$. For any $s\geq n+3$  we introduce the following integer:

\begin{equation}\label{b for l-va}
\begin{split}
& b_l:=\left\{ \begin{array}{ll}
\min\{n-1, s-n-2\}-l-1 & \textrm{ if } m_1=d-l-1 \textrm{ and }m_i= 1, i\ge 2 \\
 \min\{n,s-n-2\}-l-1 & \textrm{ elsewise,}\\\end{array}\right.
\end{split}
\end{equation}
while for $s\leq n+2$ define $b_l:=-l-1$

\begin{theorem}[$l$-very ample line bundles]\label{tva theorem X0}
Assume that $\mathcal{S}\subset\PP^n$ is a collection of points in general position.
Let $l$ be a non-negative integer. Assume that either $s\le 2n$ or  $s\ge 2n+1$ and $d$ large enough, namely 
\begin{equation}\label{tva equations X0 large s}
\sum_{i=1}^s m_i -nd \le b_l ,
\end{equation}
where $b_l$ is defined as in \eqref{b for l-va}.
Then a divisor $D$  of the form \eqref{def D}  is $l$-very ample if and only if
\begin{equation}\label{tva equations X0}
\begin{split}
&l\leq m_i, \ \forall i\in\{1,\dots,s\},\\   
&l \le d-m_i-m_j, \ \forall i,j\in\{1,\dots,s\},\ i\ne j. \\ 
\end{split}
\end{equation}
\end{theorem}

\begin{remark}
When $l=0$ ($l=1$), $l$-very ampleness corresponds to \emph{global generation}, or \emph{spannedness}
(resp. very ampleness).
\end{remark}

\begin{corollary}[Globally generated line bundles]\label{gg corollary}
In the same notation of Theorem \ref{tva theorem X0},
assume that either $s\le 2n$ or  $s\ge 2n+1$ and
$$\sum_{i=1}^sm_i -nd \le b_0.$$
Then $D$
 is globally generated if and only if
\begin{equation}\label{gg equations X0}
\begin{split}
& 0 \leq m_i, \ \forall i\in\{1,\dots,s\},\\   
& 0 \le d-m_i-m_j, \ \forall i,j\in\{1,\dots,s\},\ i\ne j.\\ 
\end{split}
\end{equation}
\end{corollary}

\begin{corollary}[Very ample line bundles]\label{va corollary}
In the same notation of Theorem \ref{tva theorem X0},
assume that either $s\le 2n$ or  $s\ge 2n+1$ and
$$\sum_{i=1}^sm_i -nd \le b_1.$$
Then $D$ is very ample if and only if
\begin{equation}\label{va equations X0}
\begin{split}
&1\leq m_i, \ \forall i\in\{1,\dots,s\},\\ 
& 1\le d-m_i-m_j, \ \forall i,j\in\{1,\dots,s\},\ i\ne j.\\  
\end{split}
\end{equation}
\end{corollary}

We present now two examples proving that the bound \eqref{tva equations X0 large s} of  Theorem \ref{tva theorem X0} is sharp. 
\begin{example}\label{example0}
Consider the divisor in $\PP^3$ 
$$D:=2H-E_1-\ldots-E_8.$$

In this case $n=3$, $d=2$, $s=8$, $l=0$ and is easy to check that $D$ satisfies equation \eqref{tva equations X0}.
However, since the bound $b_l=1$ the divisor $D$ fails to satisfy inequality \eqref{tva equations X0 large s} so hypothesis of Theorem \ref{tva theorem X0} does not apply to $D$.
In fact, divisor $D$ has a curve of degree $4$ as base locus so it is not globally generated.
\end{example}

\begin{example}\label{example}
Let us consider the anticanonical divisor of the blown-up $\PP^2$ in eight points in general position
$$D:=3H-E_1-\ldots - E_8.$$
In this case $n=2$, $s=8$, $d=3$ and $l=0$ it satisfies  equation \eqref{tva equations X0} but it fails \eqref{tva equations X0 large s} since $b_l=1$.
Sections of $D$ correspond to planar cubics passing through eight simple points. It is well-known that all such cubics meet at a ninth point, therefore $D$ is not a globally generated divisor. However, $D$ is nef. 
\end{example}

\subsection{Some technical lemmas}
In this section we prove a series of technical lemmas that will be useful in the proofs of the main theorem, Theorem \ref{tva theorem X0}, that will be given in Section \ref{main proof section}. These will also justify the integer $b_l$ appearing in  \eqref{b for l-va} and explain why we cannot obtain a better bound than \eqref{b for l-va}.

\begin{lemma}\label{bound} 
Let $D$  be the divisor defined in \eqref{def D}  and assume 
that \eqref{tva equations X0} holds.  Then $s_D(d)=0$ unless $l=0$, $s=1$ and $m_1=d$,  in which case $s_D(d)=1$.
\end{lemma}
\begin{proof}
Assume $s_D(d)\neq 0$; in particular $m_1=d$. Equations \eqref{tva equations X0} imply
$d\geq m_1+m_i+l=d+m_i+l$,  that gives $m_i=0$ for all $i\geq 1$ and $l=0$. In this case $D=dH-dE_1$ and $s_D(d)=1$. 
\end{proof}

Let $F$ be the divisor obtaining by subtracting a sum of $l+1$ exceptional divisors $E_i$, with repetitions allowed, from $D$, 
namely 
\begin{equation}\label{divisor F}
F:=D-\sum_{i=1}^s\epsilon_i E_i, \quad \epsilon:=\sum_{i=1}^s\epsilon_i=l+1,
\end{equation}
where the $\epsilon_i$'s are positive integers.

\begin{lemma}\label{bound for F} 
Let $D$ and $F$ be divisors defined respectively as  in \eqref{def D} and  \eqref{divisor F} and assume that $D$ satisfies 
\eqref{tva equations X0 large s} and
 \eqref{tva equations X0}, then $F$ satisfies \eqref{cond vanishings}.
\end{lemma}

\begin{proof}

We first proof the following claim:  $s_F(d)=0$ unless  $m_2=\ldots=m_s=1$, $m_1=d-l-1$ and $\epsilon=\epsilon_1=l+1$, in which case $s_F(d)=1$.

In order to prove the claim, assume first that $\epsilon_i\le l$ for all $i$ (in particular $l\ge 1$).  Then  \eqref{tva equations X0}   implies $m_i\ge 1$ and $d-m_i\geq m_j+l $ for every $i,j$. Therefore $s_F(d)=0$.
Otherwise, after permuting the indices if necessary, assume  that $\epsilon_1=l+1$. Then \eqref{tva equations X0}  gives $d-(m_1+l+1)\geq m_i-1\geq 0$.
Notice that $s_F(d)=0$, unless  $m_1+l+1=d$. In the latter case we have  $m_i=1$, for all $i\geq 2$, therefore we obtain
$D=dH-(d-l-1)E_1-E_2-\ldots -E_s$ and 
$F=dH-dE_1-E_2-\ldots -E_s$.

Finally, using the claim  we compute the integer $b(F)$ defined in  \eqref{b bound} for the divisor $F$: we obtain  $b(F)=\min\{n,s-n-2\}$,
unless divisor $D$ has $m_1=d-l-1$ and $m_2=1$ and in this case it is
$b(F)=\min \{n-1, s-n-2\}.$ It is now a straightforward computation to prove that  the statement holds. 
\end{proof}

Let us now introduce the following divisors $G_i$, $i=0,1$, where $\bar{s}_i=\min\{s,n-i\}$, by subtracting $l+1$ times from $D$ the strict transform on the blow-up of a hyperplane of $\PP^n$ containing $p_1,\dots, p_{\bar{s}}$:
\begin{equation}\label{divisor Gi}
G_j:=(d-l-1)H-\sum_{i=1}^{\bar{s}_j}(m_i-l-1)E_i-\sum_{i=\bar{s}_j+1}^sm_i E_i.
\end{equation}

\begin{lemma}\label{bound for Gi} 
Let $D$ and $G_j$ be divisors defined respectively as  in \eqref{def D} and  \eqref{divisor Gi} and assume that $D$ satisfies 
\eqref{tva equations X0 large s} and
 \eqref{tva equations X0}.Then $G_j$ satisfies \eqref{cond vanishings}, for $j=0,1$.
\end{lemma}
\begin{proof}
If $s_D(d)>0$, then by Lemma \ref{bound} we have that $D=d(H-E_1)$. Therefore $\bar{s}_0=\bar{s}_1=s=1$, $G_j=(d-l-1)(H-E_i)$, $s_{G_j}(d-l-1)=1$ and $G_j$ obviously satisfies \eqref{cond vanishings}, for  $j=0,1$.
Therefore we can assume that $s_D(d)=0$.

\vskip.3cm

 If $m_i<d-l-1$ for all $i$'s, then obviously $s_{G_j}(d-l-1)=0$. 
Let us write 
$$G_j=d'H-\sum_{i=1}^dm'_iE_i:=(d-l-1)H-\sum_{i=1}^{\bar{s}_j}(m_i-l-1)E_i-\sum_{i=\bar{s}_j+1}^s m_i E_i.
$$
First of all take $j=0$. 
It is an easy computation to verify that $G_0$ satisfies \eqref{cond vanishings}.
 Indeed, if $\bar{s}=s<n$ we have 
$$\sum_{i=1}^{\bar{s}}m'_i-nd'=\sum_{i=1}^{\bar{s}}(m_i-l-1)-n(d-l-1)\le0,$$ 
because $m_i\le d$. Otherwise, if $\bar{s}=n\le s$, we compute
$$
\sum_{i=1}^s m'_i-nd'=\sum_{i=1}^{\bar{s}}(m_i-l-1)+\sum_{i=\bar{s}+1}^s m_i-n(d-l-1)=
\sum_{i=1}^sm_i-nd.$$ 
The above number is bounded above by $0$ whenever $s\le 2n$, and by $\min\{n,s-n-2\}-l-1$ whenever $s\ge 2n+1$, by the hypotheses.
Moreover, in all cases one has $m'_i+m'_j-d'\le1$, for all $i\neq j$. 

Now,   take $j=1$. 
We verify that $G_1$ satisfies \eqref{cond vanishings} with a similar computation.
Indeed, if
 $\bar{s}=s<n-1$ then it is the same computation as before. While if  $\bar{s}=n-1\le s$ we have
 $$
\sum_{i=1}^sm'_i-nd'= \sum_{i=1}^{\bar{s}}(m_i-l-1)+\sum_{i=\bar{s}+1}^sm_i-n(d-l-1)=\sum_{i=1}^sm_i-nd+l+1.
 $$
The number on the right hand side of the above expression is bounded above by $0$ is $s\le 2n$ and by and by $\min\{n,s-n-2\}-l-1$ if $s\ge 2n+1$. 

\vskip.3cm

Finally, assume that $m_i=d-l-1$ for some $i$ and assume, without loss of generality, that $i=1$. In this case we have $m_i=1$ for all $i>0$, see the proof of Lemma \ref{bound for F}, and  $s_{G_j}(d-l-1)=1$. In is easy to verify that $G_j$ satisfies  \eqref{cond vanishings}.
 \end{proof}

\subsection{Proof of Theorem \ref{tva theorem X0}}\label{main proof section}

In order to give a proof of Theorem \ref{tva theorem X0}, we first give the following vanishing theorem, that has its own intrinsic interest. 

Let $\II_{\{q^{l+1}\}}$ denote the ideal sheaf of the fat point of multiplicity $l+1$ supported at $q\in \PP^n$.

\begin{theorem}\label{toric tensored} 
In the same notation as Theorem \ref{tva theorem X0},
fix integers $d,m_1,\dots,m_s,l\ge0$, $s\ge1$.
Assume that either $s\le 2n$ or that $s\ge 2n+1$ and  that
 \eqref{tva equations X0 large s} is satisfied.
Moreover, assume that 
\begin{equation}\label{inequalities}
\begin{split}
&l\leq m_i, \ \forall i\in\{1,\dots,s\},\\   
&l \le d-m_i-m_j, \ \forall i,j\in\{1,\dots,s\},\ i\ne j. 
\end{split}
\end{equation}
Then $h^1(D\otimes \II_{\{q^{l+1}\}})=0$ for any $q\in X_{s}$. 
\end{theorem}

\begin{proof}
Case (1). 
Assume first of all that  $q\in E_i$, for some $i\in\{1,\dots,s\}$. 
We claim that 
\begin{equation}\label{transferring h^1}
h^1(D\otimes \II_{\{q^{l+1}\}})\le h^1(D-(l+1)E_i).
\end{equation} Hence
we conclude because the latter vanishes, by  Theorem \ref{vanishing theorems} and Lemma \ref{bound for F} with $F=D-(l+1)E_i$. 
We now prove that \eqref{transferring h^1} holds. Let $\pi$ be the blow-up of $X_{s}$ at $q\in E_i$ and let $E_q$ be the exceptional divisor created. 
By the \emph{projection formula} we have $H^i(D\otimes\II_{\{q^{l+1}\}})\cong H^i(\pi^*(D)-(l+1)E_q)$.
For $l=0$, consider the  exact sequence 
\begin{equation}\label{seq l=0}
0\to \pi^*(D)-\pi^*(E_i)\to\pi^*(D)-E_q\to (\pi^*(D)-E_q)|_{\pi^*(E_i)-E_q}\to0.
\end{equation}
 Notice that $\pi^*(E_i)-E_q$ is the blow-up of $E_i\cong\PP^{n-1}$ at the point $q$:
 denote by $h,e_q$ the generators of its Picard group.
We have $(\pi^*(D)-E_q)|_{\pi^*(E_i)-E_q}\cong m_ih-e_q$,
 in particular it has vanishing first cohomology group. 
Hence, looking at the long exact sequence in cohomologies associated with \eqref{seq l=0}, 
one gets that the map $$H^1(\pi^*(D)-\pi^*(E_i))\to H^1(\pi^*(D)-E_q)$$ is surjective, therefore $h^1(\pi^*(D)-\pi^*(E_i))\ge h^1(\pi^*(D)-E_q)$. Finally, 
by the projection formula one has $H^i(\pi^*(D)-\pi^*(E_i))=H^i(D-E_i)$, so we conclude. For $l\ge 1$, one can iterate $l$ times the above argument and conclude.

\vskip.3cm 

Case (2). Assume $q\in X_{s}\setminus\{E_1,\dots,E_s\}$. Hence  $q$  is the pull-back of a point $q'\in \PP^n\setminus\{p_1,\dots,p_s\}$.

We will prove the statement by induction on $n$. The case $n=1$ is obvious. Indeed, any such $D\otimes\II_{\{q^{l+1}\}}$ corresponds to a linear series on the projective line given by three points whose sum of the multiplicities is bounded above as follows $m_1+m_2+(l+1)\le d+1$. Hence the first cohomology group vanishes. From now on we will assume $n\ge2$.

\vskip.3cm 

Recall that a set of  points $\mathcal{S}$ of $\PP^n$ is said to be in 
\textit{linearly general position} if for each integer $r$ we have $\sharp (S\cap L) \le r+1$, for all $r$-dimensional linear subspaces $L$ in $\PP^n$. 

\vskip.3cm

Case (2.a).
Assume first that the points in $\mathcal{S}\cup\{ q'\}$ are not in linearly general position in $\PP^n$.
If $s\ge n$, $q'$ lies on a hyperplane $H$ of $\PP^n$ spanned by $n$ points of $\mathcal{S}$. Reordering the points if necessary, assume that $q'\in H:=\langle p_1,\dots,p_n\rangle$.
If $s< n$, let $H$ be any hyperplane containing $\mathcal{S}\cup\{q'\}$.
Let $\bar{H}$ denote the pull-back of $H$ on $X_{s}$. Notice that 
$\bar{H}$ is isomorphic to the space $\PP^{n-1}$ blown-up at $\bar{s}:=\min\{s,n\}$ distinct points in general position,  so that we can write that 
$\bar{H}\cong X^{n-1}_{\bar{s}}$. Its Picard group is generated by $h:=H|_H$, $e_i:=E_i|_{H}$. 
As a divisor, we have $\bar{H}=H-\sum_{i=1}^{\bar{s}} E_i.$
Consider the restriction exact sequence of line bundles
\begin{equation}\label{restriction to hyperplane}
0\to (D-\bar{H})\otimes\II_{\{q^{l}\}}\to D\otimes\II_{\{q^{l+1}\}} \to (D\otimes\II_{\{q^{l+1}\}})|_{\bar{H}}\to0.
\end{equation}
We iterate this restriction procedure $l+1$ times.
The restriction of the $(\lambda+1)$st exact sequence, $0\le \lambda\le l$,
is the complete linear series on $X^{n-1}_{\bar{s}}$ given by 
\begin{equation}\label{restriction lambda}
\left((d-\lambda)h-\sum_{i=1}^{\bar{s}}(m_i-\lambda) e_i\right)\otimes\II_{\{q^{l+1-\lambda}\}|_H}.
\end{equation}
We leave it to the reader to verify that it satisfies the hypotheses of the theorem, for every $0\le\lambda\le l$. Hence we conclude, by induction on $n$, that the first cohomology group vanishes.

Moreover, the (last) kernel is the line bundle associated with the following divisor:
\begin{equation}\label{last kernel}
d'H-\sum_{i=1}^dm'_iE_i:=(d-l-1)H-\sum_{i=1}^{\bar{s}}(m_i-l-1)E_i-\sum_{i=\bar{s}+1}^s m_i E_i,
\end{equation}
that equals the divisor $G_0$ introduced in \eqref{divisor Gi} for $i=0$. It satisfies the condition of Theorem \ref{vanishing theorems} by Lemma 
\ref{bound for Gi}, hence we conclude in this case. 

\vskip.3cm 

Case (2.b). 
Lastly, assume that $\mathcal{S}\cup\{ q'\}$ is in linearly general position in $\PP^n$.
If $s\ge n-1$, let $H$ denote the hyperplane $\langle p_1,\dots,p_{n-1},q'\rangle$.  If $s< n-1$, let $H$ be any hyperplane containing $\mathcal{S}\cup\{q'\}$. In both cases such an $H$ exists by the assumption that points of
 $\mathcal{S}$ are in general position.
As in the previous case, let $\bar{H}$ denote the pull-back of $H$ on $X_{s}$.
 It is isomorphic to the space $\PP^{n-1}$ blown-up at $\bar{s}:=\min\{s,n-1\}$ distinct points in general position, 
that we  may denote by $\bar{H}\cong X^{n-1}_{\bar{s}}$. 

We iterate the same restriction procedure shown in \eqref{restriction to hyperplane} $l+1$ times as before.
As before the restriction of the $(\lambda+1)$st exact sequence, that is of the form \eqref{restriction lambda} with $\bar{s}$ differently defined here, verifies the hypotheses of the theorem,  so  it has vanishing first cohomology group by induction on $n$.

Furthermore, the (last) kernel, that is in the shape \eqref{last kernel}, with $\bar{s}, d', m'_i$ as defined here is the divisor $G_1$ defined in 
\eqref{divisor Gi} with $i=1$.  It satisfies the condition of Theorem \ref{vanishing theorems} by Lemma 
\ref{bound for Gi}.  
 Indeed, if
 $\bar{s}=s<n-1$ then it is the same computation as before. While if  $\bar{s}=n-1\le s$ we have
 $$
\sum_{i=1}^sm'_i-nd'= \sum_{i=1}^{\bar{s}}(m_i-l-1)+\sum_{i=\bar{s}+1}^sm_i-n(d-l-1)=\sum_{i=1}^sm_i-nd+l+1.
 $$
The number on the right hand side of the above expression is bounded above by $0$ is $s\le 2n$ and by $b$ if $s\ge 2n+1$. This concludes the proof. 

\end{proof}

Let $\ls=\ls_{n,d}(m_1,\dots,m_s)$ be the linear system of the form \eqref{linear system}.

\begin{corollary}\label{cor}
Assume that $\ls=\ls_{n,d}(m_1,\dots,m_s)$ satisfies the conditions of  Theorem \ref{toric tensored}. Then the linear system of elements of $\ls$ that vanish with multiplicity $l+1$ at an arbitrary extra point, $\ls_{n,d}(m_1,\dots,m_s,l+1)$, is non-special.

\end{corollary}
\begin{proof}
The projection formula implies that, for all $i\ge 0$,  $H^i(X_{s}, D\otimes\II_{\{q^{l+1}\}})\cong H^i(\PP^n,\ls_{n,d}(m_1,\dots,m_s,l+1))$. Therefore $\ls_{n,d}(m_1,\dots,m_s,l+1)$ has the expected dimension.
\end{proof}

Before we proceed with the proof of the main result, Theorem \ref{tva theorem X0}, we need the following lemmas.

\begin{lemma}\label{flat degeneration}
Let $X$ be a complex projective smooth variety and $\mathcal{O}_X(D)$ a line bundle. Let $Z$ be a $0$-dimensional subscheme of $X$ and let $Z_0$ be a flat degeneration of $Z$. 
Then $h^1(\mathcal{O}_X(D)\otimes\II_Z)\le h^1(\mathcal{O}_X(D)\otimes\II_{Z_0})$.
\end{lemma}
\begin{proof}
It follows from the property of \emph{upper semicontinuity} of cohomologies, see e.g, \cite[Sect. III.12]{hartshorne}. 
\end{proof}

\begin{lemma}\label{inclusion}
In the same notation of Lemma \ref{flat degeneration}, let $Z_1\subseteq Z_2$ be an inclusion of $0$-dimensional schemes.
Then $h^1(\mathcal{O}_X(D)\otimes\II_{Z_1})\le h^1(\mathcal{O}_X(D)\otimes\II_{Z_2})$.
\end{lemma}
\begin{proof}
If $Z_1=Z_2$ then equality obviously holds. We will assume $Z_1\subsetneq Z_2$.
Consider the following short exact sequence
$$0\stackrel{}{\rightarrow}\mathcal{O}_X(D)\otimes\II_{Z_2}\stackrel{}{\rightarrow}\mathcal{O}_X(D)\otimes\II_{Z_1}\stackrel{}{\rightarrow}\mathcal{O}_X(D)\otimes\II_{Z_1}|_{Z_2\setminus Z_1}\stackrel{}{\rightarrow}0.$$
Consider the associated long exact sequence in cohomology. Since 
$h^1(\mathcal{O}_X(D)\otimes\II_{Z_1}|_{Z_2\setminus Z_1})=0$, the map $H^1(\mathcal{O}_X(D)\otimes\II_{Z_2})\to H^1(\mathcal{O}_X(D)\otimes\II_{Z_1})$ is surjective and this concludes the proof.
\end{proof}

Lemma \ref{flat degeneration} and Lemma \ref{inclusion} will allow to reduce the proof of $l$-very ampleness for divisors $D$ to the computation of vanishing theorems of the first cohomology group of the sheaf associated with $D$ tensored by the ideal sheaf of a collection of fat points, whose multiplicities sum up to $l+1$.

\begin{proof}[Proof of Theorem \ref{tva theorem X0}]
We first prove that  \eqref{tva equations X0} is sufficient condition for $D$ to be $l$-very ample. For every $0$-dimensional scheme $Z\subset X_{s}$ of weight 
 $l+1$, consider the exact sequence of sheaves
\begin{equation}\label{sequence for lva}
0\to \mathcal{O}_{X_{s}}(D)\otimes\II_Z\to \mathcal{O}_{X_{s}}(D)\to \mathcal{O}_{X_{s}}(D)|_Z\to0.
\end{equation}
We will prove that $h^1(X_{s},\mathcal{O}_{X_{s}}(D)\otimes\II_Z)=0$. 
This will imply the surjectivity of the map 
$H^0(X_{s},\mathcal{O}_{X_{s}}(D))\to H^0(Z, \mathcal{O}_{X_{s}}(D)|_Z)$, by taking the
long exact sequence in cohomology associated with \eqref{sequence for lva}.

Let $Z_0$ denote a flat degeneration of $Z$ with support at the union of points $q_1,\dots,q_s,q_{s+1}\in  X_{s}$, with  $q_i\in E_i$, for all $i=1,\dots,s$, and $q_{s+1}\in X_{s}\setminus\{E_1,\dots,E_s\}$. 
Since every exceptional divisor $E_i$ as well as $X_{s}\setminus\{E_1,\dots,E_s\}$ are homogeneous spaces, in order to prove that $h^1(D\otimes\II_Z)=0$ for all $Z\subset X_{s}$ $0$-dimensional schemes of weight $l+1$, by Lemma \ref{flat degeneration} 
 it is enough to prove that the same statement holds for every such $Z_0$.

Set  $\mu_i$ to be the weight of the irreducible component of $Z_0$ supported at $q_i$, 
for all $i=1,\dots,s,{s+1}$. One has that $\mu_i\ge0$, $\sum_{i=1}^{s+1}\mu_i=l+1$.
In order to prove that $h^1(D\otimes\II_{Z_0})=0$, by Lemma \ref{inclusion}, it suffices to prove the a priori stronger statement that 
 $h^1(\mathcal{O}_X(D)\otimes\II_{\bar{Z}_0})=0$, for every $\bar{Z}_0$ collection of fat points $\{q_1^{\mu_1},\dots, q_s^{\mu_s},q_{s+1}^{\mu_{s+1}}\}$. Indeed, the following containment of schemes supported at $q_i$ holds for all $i=1,\dots,s,{s+1}$: $Z_i|_{q_i}\subseteq \{q_i^{\mu_i}\}$.

Assume first of all that $\bar{Z}_0$ has support in one point $q=q_i$, namely $\mu_i=l+1$ for some $i\in\{1,\dots,s,{s+1}\}$ and $m_j=0$ for all $j\neq i$. We have $h^1(D\otimes\II_{\bar{Z}_0})=0$ by Theorem \ref{toric tensored}.
In words, we are completing the scheme $Z_0$ to a scheme $\bar{Z}_0$ that consists of a disjoint union of fat points and proving that the statement holds for the latter.

Assume now that  $\bar{Z}_0$ is supported in several points $q_1,\dots,q_s,q_{s+1}$. 
We want to prove that  $h^i(X_{s},D\otimes\II_{\bar{Z}_0})\le h^i(X_{s},\left(D-\sum_{i=1}^s \mu_i E_i\right)\otimes\II_{\{q_{s+1}^{\mu_{s+1}}\}})=0$. The first inequality follows from \eqref{transferring h^1}.  The  equality follows from Theorem \ref{toric tensored}; 
we leave it to the reader to verify that the conditions are indeed satisfied. 

\vskip.3cm

We now prove that \eqref{tva equations X0} is necessary condition for $D$ to be $l$-very  
ample, by induction on $l$. 

Let us first assume $l=0$, namely that $D$ is base point free. 
If $m_i<0$ then $m_iE_i$ would be contained in the base locus of $D$. 
If $m_i+m_j> d$ for some $i\ne j$, then the strict transform of the line $\langle p_i,p_j\rangle\subset\PP^n$ 
would be contained in the base locus of $D$. 
In both cases we would obtain a contradiction.

Assume that $l=1$, namely that $D$ is very ample. If $m_i\le0$ (or $0\le d-m_i-m_j$ for some $i\ne j$), 
then $E_i$ (resp. the strict transform of the line through $p_i$ and $p_j$) would be contracted by $D$, a contradiction.

More generally, assume that $D$ is $l$-very ample and $l\ge2$.  
Then conditions \eqref{tva equations X0} are satisfied. Indeed, if $m_i\le l-1$ for some $i$, 
we can find a $0$-dimensional scheme, $Z$, of weight $l+1$ such that
 $h^1(D\otimes\II_Z)>0$. Let $Z\subset E_i$ be an $l$-\emph{jet scheme} centred at $q\in E_i$ (see \cite{mustata}).
Consider the restriction $D\otimes\II_Z|_{E_i}\cong m_1h\otimes\II_Z$, where $h$
 is the hyperplane class of $E_i\cong\PP^{n-1}$. We have $h^1(E_i,D\otimes\II_Z|_{E_i})\ge1$, hence
$h^1(X_{s},D\otimes\II_Z)\ge1$. To see this, let $x_1\dots,x_{n-1}$ be
affine coordinates for an affine chart $U\subset E_i$ and let  $Z$ be the jet-scheme 
with support 
$q=(0,\dots,0)\in U$
given by the tangent directions up to order $l$ along $x_1$. 
The space of global sections of $D\otimes\II_Z|_{E_i}$
is isomorphic to the set of degree-$m_i$ polynomials $f(x_1,\dots, x_{n-1})$,
 whose partial derivatives $\partial^\lambda f/\partial x_1^{\lambda}$
 vanish at $q$, for $0\le \lambda\le l$. On the other hand,
$H^1(E_i,D\otimes\II_Z|_{E_i})$
is the ``space of linear dependencies'' among the $l+1$ conditions imposed by the vanishing of
 the partial derivatives 
to the coefficients of $f$. Since $m_i\leq l-1$ then $f$ is a polynomial of degree at most  $l-1$, therefore $\partial^{l} f/\partial x_1^{l}\equiv0$ for every such a polynomial, and we conclude. 

Similarly, if $d-m_i-m_j\le l-1$ for some $i,j,i\ne j$, then one finds a jet-scheme $Z$ contained in the pull-back of the line through $p_i$ and $p_j$, $L$, for which $h^1(X_{s},D\otimes\II_Z)\ge1$. Indeed, if $Z$ is such a scheme, then the restriction is
 $D\otimes\II_Z|_L\cong (d-m_i-m_j)h\otimes\II_Z|_{L}$, where in this case $h$ is  the class of a point in $L$, and $Z|_L$ is a fat point of multiplicity $l$ in $L$. One concludes by the Riemann-Roch Theorem that $h^1(L,D\otimes\II_Z|_L)\ge1$ because $\chi(L,D\otimes\II_Z|_L)=(d-m_i-m_j)-l\le-1$ and $h^0(L,D\otimes\II_Z|_L)=0$.

\end{proof}

\subsection{$l$-jet ampleness}
In \cite{bfs}, Beltrametti, Francia and Sommese introduced  notions of 
higher order embeddings, one of these
being $l$-very ampleness (Definition \ref{definition k-va}) with the aim of studying the \emph{adjoint bundle} on surfaces.

\begin{definition}
In the same notation as Definition \ref{definition k-va},
 if for any fat point $Z=\{q^{l+1}\}$, $q\in X$,  the natural restriction map to $Z$,
 $H^0(X,\mathcal{O}_X(D))\to H^0(Z, \mathcal{O}_X(D)|_{Z})$, is surjective, then $D$
 is said to be $l$-\emph{jet spanned}.

Moreover, if for any collection of \emph{fat points} 
$Z=\{q_1^{\mu_1},\dots,q_\sigma^{\mu_{\sigma}}\}$ such that $\sum_{i=1}^{\sigma}\mu_i=l+1$,
 the restriction map to $Z$ is surjective, then $D$
 is said to be $l$-\emph{jet ample}.
\end{definition}

\begin{remark}\label{jet spannedness}
Theorem \ref{toric tensored} can be restated in terms of $l$-jet spannedness. Namely
 any divisor $D$ satisfying the hypotheses is $l$-jet spanned.

\end{remark}

\begin{proposition}[{\cite[Proposition 2.2]{belsom}}]\label{jet implies very}
In the above notation, if $D$ is $l$-jet ample, then $D$ is $l$-very ample.
\end{proposition}

The converse of Proposition \ref{jet implies very} is  true for the projective space $\PP^n$ and for curves, but not in general. 
In this section we proved that the converse is true for lines bundle $\mathcal{O}_{X_s}(D)$ on $X_{s}$,
 that satisfy the hypotheses of Theorem \ref{tva theorem X0}.

\begin{theorem}\label{jet ampleness}
Assume that $s\le 2n$, or $s\ge 2n+1$ and \eqref{tva equations X0 large s}. Assume that 
 $D$ is a line bundle on $X_{s}$ of the form \eqref{def D}. The following are equivalent:
\begin{enumerate}
\item $D$ satisfies \eqref{tva equations X0};
\item $D$ is  $l$-jet ample;
\item $D$ is $l$-very ample.
\end{enumerate}
\end{theorem}
\begin{proof}
We proved that the natural restriction map of the global sections of $D$ to the any fat point of multiplicity $l+1$ is surjective
 in Theorem \ref{toric tensored}, see also Remark \ref{jet spannedness}. We showed that the same is true in the case of
arbitrary collections of fat points whose multiplicity sum up to 
$l+1$ in the first part of the proof of Theorem
\ref{tva theorem X0}.  This proves that $(1)$ implies $(2)$. Moreover, $(2)$ implies $(3)$
 by Proposition \ref{jet implies very}. 
Finally, that $(3)$ implies $(1)$ was proved in the second part of the proof of Theorem \ref{tva theorem X0}.
\end{proof}

\section{Other positivity properties of divisors on $X_{s}$}\label{other positivity}

In this section we will apply Theorem \ref{tva theorem X0} to establish further positivity
properties of divisors on $X_{s}$. All results 
we prove in this section apply to $\mathbb{Q}-$divisors on the blown-up projective space.

We recall here the notation introduced in Section \ref{lva section}: for a given divisor $D$ 
of the form $D=dH-\sum_{i=1}^sm_iE_i,$ (cfr. \eqref{def D}),
we will use the integer $b=b(D):= \min\{n -s(d); s-n-2\}$, defined in (\ref{b bound}) 
and the bound \eqref{tva equations X0 large s}:
$$
\sum_{i=1}^s m_i -nd \le b_l.
$$

 \subsection{Semi-ampleness and ampleness}\label{section semiampleness}

A line bundle is ample if some positive  power is very ample.  
It is known that for smooth toric varieties a divisor is ample if and only if is very ample and nef if and only if is globally generated.
  From Corollary \ref{gg corollary} and Corollary \ref{va corollary}, we obtain
that this holds for a small number of points $s\le 2n$ too, as well as for 
arbitrary $s$ under a bound on the coefficients.

A line bundle is called \emph{semi-ample}, or \emph{eventually free}, if some positive  power is globally generated. By \eqref{gg equations X0}, one can see that a divisor is semi-ample if and only if it is globally generated.

\begin{theorem}
\label{ample cone}
 Let $X_{s}$ be defined as in Section \ref{notations}. Assume $s\le 2n$.
\begin{enumerate}
\item
The cone of semi-ample divisors in $\Nn^1(X_{s})_{\mathbb{R}}$  is given by 
\eqref{gg equations X0}.
\item 
The cone of ample divisors in $\Nn^1(X_{s})_{\mathbb{R}}$ is given by \eqref{va equations X0}.
\end{enumerate}

Assume $s\ge 2n+1$.
\begin{enumerate}
\item   Divisors satisfying 
\eqref{tva equations X0 large s} with $l=0$ 
are semi-ample if and only if \eqref{gg equations X0}.
\item 
Divisors satisfying \eqref{tva equations X0 large s} with $l=1$
are ample if and only if \eqref{va equations X0}.
\end{enumerate}

\end{theorem}

\subsection{Nefness}\label{section nefness}

For any projective variety, Kleiman \cite{kleiman} showed that a divisor is ample if and only if its numerical equivalence class lies in the interior of the nef cone
 (see also \cite[Theorem 1.4.23]{lazarsfeld}).
 
For a line bundle, being  generated by the global sections implies being nef, but the opposite is not true in general,
see e.g. Example \ref{example}. However for line bundles on $X_{s}$, with $s\le 2n$,
 or with arbitrary $s$ under a bound on the coefficients, 
 these two properties are equivalent.

\begin{theorem}\label{nef cone}
In the same notation as Theorem \ref{tva theorem X0}, assume that for $D$ of the form 
\eqref{def D} we have that either $s\le 2n$ or  $s\ge 2n+1$ and 
\eqref{tva equations X0 large s} with $l=0$
is satisfied.
Then $D$  is nef if and only if is globally generated.

\end{theorem}

\begin{proof}
If $D$ is nef, then for effective $1$-cycle $C$, $D\cdot C\geq 0$. 
In particular the divisor $D$ intersects positively the classes of lines through
 two points and classes of lines in the exceptional divisors. This means inequalities 
\eqref{gg equations X0} hold and therefore the divisor $D$ is globally generated by 
Corollary \ref{gg corollary}.
\end{proof}

\begin{remark}
If $s\le 2^n$, the nef cone of $X_{s}$ is given by  \eqref{gg equations X0}. In fact, if $X$ is the blow-up of $\PP^n$ in $s=2^n$ points obtained as intersection of $n$ quadrics of $\PP^n$, then the class of a quadric on $X$ through $2^n$ points is nef (because the corresponding linear system is base point free). Therefore the same holds for general points.

Notice that for $s\le 2n$, the description of the nef cone also follows from Theorem \ref{nef cone}. 
\end{remark}

\begin{corollary}\label{nef = sa}
The nef cone and the cone of semi-ample divisors on $X_{s}$, for $s\le 2n$, coincide. 
\end{corollary}

\begin{corollary}\label{mori cone}
The Mori Cone of curves of the blown-up $\PP^n$ in $s\leq 2^n$ 
points,  $\overline{\NE}(X_{s})$, 
  is generated by the classes of lines through two points and the classes of lines in the exceptional divisors.
\end{corollary}

\subsection{Fujita's conjectures for the blown-up $\PP^n$ in points}

\begin{conjecture}[Fujita's conjectures, \cite{fujita}]\label{fujita}
Let $X$ be an $n$-dimensional projective algebraic variety, smooth or with mild singularities. 
Let $K_X$ be the canonical divisor of $X$ and  $D$ an ample divisor on $X$. Then the following holds.
\begin{enumerate}
\item For $m\ge n+1$, $mD+K_X$ is globally generated.
\item For $m\ge n+2$, $mD+K_X$ is very ample.
\end{enumerate}
\end{conjecture}

Fujita's conjecture hold on  any smooth variety where all nef divisors are semi-ample. In particular it holds on $X_s$ with $s\le 2n$. 
\begin{proposition}\label{fujita answer}
Let $X_{s}$ be the blown-up $\PP^{n}$ at $s$ points in  general position with $s\le 2n$. Conjecture \ref{fujita} holds for $X_{s}$.   
\end{proposition}

\begin{proof}

For $X_{s}$, $s\le 2n$, global generation (very ampleness) is equivalent to nefness (resp. ampleness), by Theorem \ref{ample cone} and Theorem \ref{nef cone}. This concludes the proof.
\end{proof}

Using the results from this article, we can extend the above to an infinite family of divisors with arbitrary $s$.

\begin{proposition}\label{fujita answer2}
Let $X_{s}$ be the blown-up $\PP^{n}$ in an arbitrary number of points in 
general position, $s$, and let $D$ be a divisor on $X_{s}$ such that 
\begin{equation}\label{zero bound}
\sum_{i=1}^s m_i\leq nd
\end{equation}
Then Conjecture \ref{fujita} holds for  $D$.
\end{proposition}

\begin{proof}
It is enough to consider the case $s\geq 2n+1$. 
Write $X=X_{s}$. Notice that the divisor $mD+K_X$ has the following properties
\begin{align*}
\sum_{i=1}^s (mm_i - n+1) - n(md-n-1) &=
m(\sum_{i=1}^s m_i - n d)+n(n+1) - s(n-1)\\
&\leq n(n+1)-(2n+1)(n-1)\\
&=(-n^2+n)+(n+1)\\
&\leq -2+n+1\\
&=n-1.\\
\end{align*}

Notice that $b(mD+K_X)=n-1$, for $s=2n+1$ and $b(D)=n$ for $s\geq 2n+2$, using 
the definition \eqref{b bound}.
Therefore $mD+K_X$ satisfies conditions of
Theorem \ref{ample cone}. We now leave it to the reader to check that if $D$ is ample then the divisor $mD+K_X$ satisfies conditions \eqref{gg equations X0} and \eqref{va equations X0}.

\end{proof}

\end{document}